\documentclass[12pt]{amsart}   
\usepackage{amssymb,hyperref}
\usepackage{tikz}
\usepackage[all]{xy}

\theoremstyle{plain}
\newtheorem{theorem}{Theorem}[section]

\newtheorem{lemma}[theorem]{Lemma}
\newtheorem{proposition}[theorem]{Proposition}
\newtheorem{corollary}[theorem]{Corollary}

\theoremstyle{definition}
\newtheorem{definition}[theorem]{Definition}

\newtheorem{remark}[theorem]{Remark}

\newcommand{\A}{\ensuremath{\mathcal{A}}}
\newcommand{\B}{\ensuremath{\mathcal{B}}}
\newcommand{\CC}{\ensuremath{\mathcal{C}}}
\newcommand{\D}{\ensuremath{\mathcal{D}}}
\newcommand{\V}{\ensuremath{\mathcal{V}}}
\DeclareMathOperator{\BSub}{BSub}

\DeclareMathOperator{\oDir}{Dir}

\DeclareMathOperator{\Proj}{Proj}

\newcommand{\mc}{\mathcal}

\newcommand{\sa}{\ensuremath{\mathrm{sa}}}

\pgfdeclarelayer{dotlayer}
\pgfdeclarelayer{linelayer}
\pgfsetlayers{main,linelayer,dotlayer}
\makeatletter
\pgfkeys{%
  /tikz/on layer/.code={
    \pgfonlayer{#1}\begingroup
    \aftergroup\endpgfonlayer
    \aftergroup\endgroup
  },
  /tikz/node on layer/.code={
    \gdef\node@@on@layer{%
      \setbox\tikz@tempbox=\hbox\bgroup\pgfonlayer{#1}\unhbox\tikz@tempbox\endpgfonlayer\pgfsetlinewidth{0.7pt}\egroup}
    \aftergroup\node@on@layer
  },
  /tikz/end node on layer/.code={
    \endpgfonlayer\endgroup\endgroup
  }
}
\def\node@on@layer{\aftergroup\node@@on@layer}
\makeatother
\tikzstyle{dot}=[circle, draw=black!70, fill=black!20, inner sep=.25ex, node on layer=dotlayer]
\tikzstyle{Dot}=[circle, draw=black, fill=black!70, inner sep=.25ex, node on layer=dotlayer]
\tikzstyle{thinline}=[-, draw=black!20, line width=.5pt, node on layer=linelayer]
\tikzstyle{line}=[-, draw=black!50, line width=.5pt, node on layer=linelayer]
\tikzstyle{Line}=[-, draw=black, line width=1pt, node on layer=linelayer]

\title{Orthogeometries and AW*-algebras}

\author{John Harding}
\address{New Mexico State University}
\email{hardingj@nmsu.edu}

\author{Bert Lindenhovius}
\address{Tulane University}
\email{alindenh@tulane.edu} 

\begin{document}

\begin{abstract}
Based on results of \cite{HHLN} we give a connection between the category of AW*-algebras and their normal Jordan homomorphisms and a category $\mathbf{COG}$ of \emph{orthogemetries}, which are structures that are somewhat similar to projective geometries, consisting of a set of points and a set of lines, where each line contains exactly 3 points. They are constructed from the commutative AW*-subalgebras of an AW*-algebra that have at most an 8-element Boolean algebra of projections. Morphisms between orthogemetries are partial functions between their sets of points as in projective geometry. The functor we create $\A_*:\mathbf{AW}_J^*\to\mathbf{COG}$ is injective on non-trivial objects, and full and faithful with respect to morphisms that do not involve type $I_2$ factors. 
\end{abstract}

\maketitle

\section{Introduction}

The subject of this note is the poset $\B(A)$ (denoted in \cite{HHLN} by $\BSub(A)$) of Boolean subalgebras of an orthomodular lattice $A$, and the poset $\A(M)$ of abelian subalgebras of an AW*-algebra $M$. Study of such posets grew from the topos approach to the foundations of physics initiated by Isham and others \cite{DoringIsham,Isham} where the topos of presheaves on such posets were the primary object of study, providing ``classical snapshots'' of a quantum system. 

A number of papers \cite{Doring,Hamhalter,HHLN,HardingNavara,Lindenhovius} have considered aspects of such posets and their relationships to the originating orthomodular structure or operator algebra. In \cite{HHLN} is given a direct method to reconstruct an orthoalgebra $A$ from the poset $\B(A)^*$ of Boolean subalgebras of $A$ having at most 16 elements. Such posets arising this way are called {\em orthohypergraphs}. They behave much like projective geometries, and morphisms between them are defined much as in the case of classical projective geometry \cite{Faure}. Hypergraphs arising in this way from an orthoalgebra are characterized, and it is shown that there is a ``near'' categorical equivalence between the category of orthoalgebras and that of such orthohypergraphs. This can be viewed as an extension of the technique of Greechie diagrams \cite{GreechieD}. 

In this note we specialize and adapt results of \cite{HHLN} to the setting of orthomodular lattices and the orthomodular lattice homomorphisms between them. We then use a version of Dye's theorem \cite{Hamhalter2,HeunenReyes} to lift results to the AW*-algebra setting. We provide a functor from the category of AW*-algebras and their normal Jordan $*$-homomorphisms to the category of orthohypergraphs and their normal hypergraph morphisms that is injective on objects and full and faithful on morphisms provided no type $I_2$ summands are present. 

This note is arranged in the following way. In the second section we adapt and simplify results of \cite{HHLN} to apply to orthomodular posets and lattices, and in the third section we lift these results to the AW*-setting using Dye's theorem \cite{Hamhalter2,HeunenReyes}. Results, notation, and terminology of \cite{HHLN} will be assumed. For general reference on orthomodular structures see \cite{DvurPulm:Trends,Kalmbach}.

\section{hypergraphs of orthomodular posets and lattices}

In \cite{HHLN} hypergraphs of orthoalgebras were characterized and a near equivalence between the category of orthoalgebras and orthohypergraphs was given. Here these results are specialized and simplified for application to orthomodular posets and orthomodular lattices. Although we recall a few basics, the reader should consult \cite{HHLN} for background and notation. We begin with the following from \cite[Def.~6.1]{HHLN}. 

\begin{definition}
A {\em hypergraph} is a triple $\mathcal{H}=(P,L,T)$ consisting of a set $P$ of {\em points}, a set $L$ of {\em lines}, and a set $T$ of {\em planes}. A line is a set of 3 points, and a plane is a set of 7 points where the restriction of the lines to these 7 points is as shown below. 
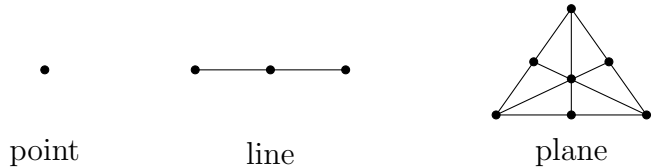
\vspace{2ex}

\begin{center}
\begin{tikzpicture}[scale = 1]
\draw[fill] (-6,.6) circle(1.5pt); \draw[fill] (-4,.6) circle(1.5pt);\draw[fill] (-3,.6) circle(1.5pt);\draw[fill] (-2,.6) circle(1.5pt);
\draw[thin] (-4,.6)--(-2,.6);
\node at (-6,-.5) {point}; \node at (-3,-.5) {line}; \node at (1,-.5) {plane};
\draw[fill] (0,0) circle(1.5pt); \draw[fill] (1,0) circle(1.5pt); \draw[fill] (2,0) circle(1.5pt); \draw[fill] (1,1.414) circle(1.5pt); \draw[fill] (1,.48) circle(1.5pt); \draw[fill] (.5,.707) circle(1.5pt); \draw[fill] (1.5,.707) circle(1.5pt); 
\draw[thin] (0,0)--(2,0)--(1,1.414)--(0,0)--(1.5,.707); \draw[thin] (1,1.414)--(1,0); \draw[thin] (.5,.707)--(2,0);
\end{tikzpicture}
\end{center}
\end{definition}

From an orthoalgebra $A$, one creates its hypergraph $\mathcal{H}(A)$ whose points, lines, and planes are the Boolean subalgebras of $A$ having 4, 8, and 16 elements, respectively \cite[Defn.~6.3]{HHLN}. 

When dealing with hypergraphs, a plane is represented as a certain configuration of 7 points and 6 lines. However, in the case of hypergraphs arising as $\mathcal{H}(A)$ for an orthoalgebra $A$, not every such configuration of points and lines is indeed a plane. So in general, planes must be separate entities. On the other hand \cite[Prop.~6.7]{HHLN} any such configuration of points and lines in a hypergraph $\mathcal{H}(A)$ for an orthomodular poset $A$ will constitute a plane. So when working with orthomodular posets, we aim to treat planes as a derived notion. This will be our setting, and we introduce new terminology for it. 

\begin{definition}
A pre-orthogeometry $\mathcal{G}=(P,L)$ is a set $P$ of points and a set $L$ of lines where each line is a set of 3 points such that any two points are contained in at most one line. A subspace of a pre-orthogeometry is a set $S$ of points such that whenever 2 points of a line belong to $S$, then the third point of the line also belongs to $S$. 
\end{definition}

Let $\mathcal{H}(A)$ be the hypergraph of an orthoalgebra $A$ and $p,q$ be distinct points of it. So $p,q$ are 4-element Boolean subalgebras of $A$. Then $p,q$ lie on a line $\ell$, that is an 8-element Boolean subalgebra of $A$, iff there are orthogonal elements $a,b\in A\setminus\{0,1\}$ such that $p=\{0,a,a',1\}$ and $q=\{0,b,b',1\}$. In this case, $\{0,a,a',b,b',a\vee b,a'\wedge b',1\}$ is the only 8-element Boolean subalgebra of $A$ that contains $p,q$. Thus all points of $\mathcal H(A)$ are contained in at most one line. 
	
For an arbitrary hypergraph $\mathcal{H}=(P,L,T)$ such that any two points are contained in at most one line, $\mathcal H$ gives a pre-orthogeometry $\mathcal{H}_*=(P,L)$ by forgetting its set of planes. For a pre-orthogeometry $\mathcal{G}=(P,L)$, construct a hypergraph $\mathcal{G}^*=(P,L,T)$ whose planes are subspaces of 7 points such that there are exactly 6 lines that contain at least two of these points, and these 7 points and 6 lines form a configuration that could be a plane. We always have $(\mathcal{G}^*)_*=\mathcal{G}$, and if $\mathcal{H}$ is the hypergraph of an orthomodular poset we have $(\mathcal{H}_*)^*=\mathcal{H}$ \cite[Prop. 6.7]{HHLN}. We next provide a link between pre-orthogeometries and the orthodomains that are described in \cite[Defn.~4.3]{HHLN}. This allows access to results of \cite{HHLN}. 

\begin{lemma} \label{lk}
Let $\mathcal{G}=(P,L)$ be a pre-orthogeometry. Then considering $\mathcal{G}^*$ as a poset by adding a new bottom $\bot$ and with the natural order of inclusion among the points, lines, and planes, we have that $\mathcal{G}^*$ is an orthodomain. 
\end{lemma}

\begin{proof}
In this poset $\mathcal{G}^*$ the points are atoms, each line is the join of the atoms it contains, and each plane is the join of the atoms it contains. So it is atomistic, and since it is of finite height, it is directedly complete and the atoms are compact. By construction, for each element $x\in\mathcal{G}^*$ the downset $\downarrow x$ is a Boolean domain. Finally, if $x,y$ are distinct atoms of $\mathcal{G}^*$ and are covered by $w$, then $x,y$ are points of $\mathcal{G}$ and $w$ is a line of $\mathcal{G}$ that contains $x,y$. Since two lines of $\mathcal{G}$ that contain the same two points must be equal, there is no other element of $\mathcal{G}^*$ of height 2 containing $x,y$. Since any plane of $\mathcal{G}^*$ that contains $x,y$ must contain the line $w$ they span, it follows that $w$ is the join $x\vee y$. Thus $\mathcal{G}^*$ satisfies the conditions of \cite[Defn.~4.3]{HHLN}, so is an orthodomain. 
\end{proof}

\begin{definition}
An orthogeometry $\mathcal{G}=(P,L)$ is a pre-orthogeometry that is isomorphic to $(\mathcal{H}(A))_*$ for some orthomodular poset $A$. Thus its points and lines are configured as the 4 and 8-element Boolean subalgebras of $A$. 
\end{definition}

We aim to characterize orthogeometries in elementary terms. As in \cite{HHLN}, key is the notion of a direction. In the setting of pre-orthogeometries, this definition simplifies considerably. 

\begin{definition}
A direction $d$ for a point $p$ of a pre-orthogeometry $\mathcal{G}$ is an assignment $d(\ell)$ of either $\uparrow$ or $\downarrow$ to each line $\ell$ that contains $p$ such that two lines $\ell$ and $m$ obtain opposite values of $\uparrow$ and $\downarrow$ iff they lie as part of a subspace of $\mathcal{G}$ forming a plane with $\ell,m$ the only lines in this plane containing $p$. 
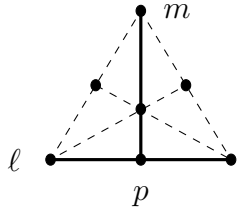
\vspace{2ex}

\begin{center}
\begin{tikzpicture}[xscale=1.2,yscale=1.4]
\draw[fill] (0,0) circle(1.5pt); \draw[fill] (1,0) circle(1.5pt); \draw[fill] (2,0) circle(1.5pt); \draw[fill] (1,1.414) circle(1.5pt); \draw[fill] (1,.48) circle(1.5pt); \draw[fill] (.5,.707) circle(1.5pt); \draw[fill] (1.5,.707) circle(1.5pt); 
\draw[thin, dashed] (0,0)--(2,0)--(1,1.414)--(0,0)--(1.5,.707); \draw[very thick] (1,1.414)--(1,0); \draw[thin, dashed] (.5,.707)--(2,0); \draw[very thick] (0,0)--(2,0);
\node at (-.4,0) {$\ell$}; \node at (1.4,1.414) {$m$}; \node at (1,-.35) {$p$};
\end{tikzpicture} 
\end{center}
\vspace{1ex}
The directions for $\mathcal{G}$ are the directions for the points of $\mathcal{G}$ and two additional directions $0,1$. 
\end{definition}

Note that since any two points are contained in at most one line, any two intersecting lines span at most one plane. 
The idea behind a direction is as follows. In the hypergraph of an orthoalgebra $A$, a point $p$ is a 4-element Boolean subalgebra $\{0,a,a',1\}$ for some $a\neq 0,1$ in $A$. A line $\ell$ containing $p$ is an 8-element Boolean subalgebra that contains $a$, and a plane containing $p$ is a 16-element Boolean subalgebra containing $a$. The direction $d$ for $p$ says whether $a$ occurs as an atom or coatom in $\ell$. If $a$ occurs as a coatom of $\ell$ and atom of $m$, then there are $b<a$ in $\ell$ and $a<c$ in $m$. Then there is a 16-element Boolean subalgebra generated by the chain $b<a<c$, and this is the plane containing $\ell$ and $m$ in the definition of a direction. 

\begin{lemma} \label{dir}
Let $\mathcal{G}=(P,L)$ be a pre-orthogeometry. 
A direction $d$ of $\mathcal{G}$ for a point $p\in P$ extends uniquely to a direction $d^*$ of the orthodomain $\mathcal{G}^*$ for $p$, and a direction $e$ of $\mathcal{G}^*$ for $p$ restricts to a direction of $\mathcal{G}$ for $p$.
\end{lemma}

\begin{proof}
By Lemma~\ref{lk} $\mathcal{G}^*$ is an orthodomain. By \cite[Defn.~4.9]{HHLN} a direction $e$ of the orthodomain $\mathcal{G}^*$ for a point $p$ is a mapping whose domain is the upper set ${\uparrow}p$. This map associates to each $y\in{\uparrow}p$ a pair of elements $e(y)=(v,w)$ called a principle pair in the Boolean domain ${\downarrow}y$. In the current circumstances this means that $e(p)=(p,p)$ and for a line $\ell$ with $p<\ell\leq y$ that $e(\ell)$ is either $(p,\ell)$ or $(\ell,p)$. Here, as in \cite[Cor.~2.21 ff]{HHLN} we denote $(p,\ell)$ by $\downarrow$ and $(\ell,p)$ by $\uparrow$. For a plane $w$ containing $p$, if $p$ lies on three lines of $w$ then $e(w)$ is either $(p,w)$ or $(w,p)$, and if $p$ lies on only two lines of $w$ then $e(w)$ is either $(\ell,m)$ or $(m,\ell)$ where $\ell,m$ are the unique lines of $w$ that contain $p$. Additionally, if $p\leq z\leq y$ and $d(y)=(u,v)$ then \cite[Defn.~4.9]{HHLN} we have $d(z)=(z\wedge u,z\wedge v)$. Finally \cite[Defn.~4.9]{HHLN} if $\ell,m$ are lines containing $p$ and $d(\ell),d(m)$ take opposite values of $\uparrow,\downarrow$, then $\ell\vee m$ exists and is a plane containing both. 

So if $e$ is a direction of the orthodomain $\mathcal{G}^*$ for $p$, by restricting $e$ to the lines of $\mathcal{G}$ that contain $p$ and for such a line $\ell$ using $\downarrow$ for $(p,\ell)$ and $\uparrow$ for $(\ell,p)$ we have the restriction of $e$ is an assignment of $\downarrow,\uparrow$ to each line of $\mathcal{G}$ containing $p$. To show this is a direction of the pre-orthogeometry $\mathcal{G}$ we must show that if $\ell,m$ are lines containing $p$, then $e(\ell)$ and $e(m)$ take opposite values of $\downarrow,\uparrow$ iff $\ell,m$ belong to a plane and are the only two lines of this plane containing $p$. If $e(\ell),e(m)$ take different values of $\downarrow,\uparrow$, then \cite[Defn.~4.9]{HHLN} gives that $w=\ell\vee m$ is a plane. If $p$ lies on three lines of this plane, we have from the first paragraph of the proof that $e(w)$ is either $(p,w)$ or $(w,p)$, and it follows that either $e(\ell)=(p,\ell)= {\downarrow}$ and $e(m)=(p,m)={\downarrow}$ or $e(\ell)=(\ell,p)={\uparrow}$ and $e(m)=(m,p)={\uparrow}$, contrary to assumption. Conversely, if $\ell,m$ belong to a plane $w$ and are the only two lines of $w$ containing $p$, the first paragraph shows $e(w)$ is either $(\ell,m)$ or $(m,\ell)$. In the first case, $e(\ell)=(\ell\wedge\ell,m\wedge\ell)=(\ell,p)={\uparrow}$ and $e(m)=(\ell\wedge m,m\wedge m)=(p,m)={\downarrow}$. The second case is symmetric, and in either case $e$ takes opposite values of $\downarrow,\uparrow$ at $\ell,m$. So $e$ restricts to a direction of $\mathcal{G}$. 

Let $d$ be a direction of $\mathcal{G}$ for $p$. We wish to extend this to a direction $e$ of $\mathcal{G}^*$. For a line $\ell$ containing $p$ set $e(\ell)=(p,\ell)$ if $d(\ell)={\downarrow}$ and set $e(\ell)=(\ell,p)$ if $d(\ell)={\uparrow}$. Suppose $w$ is a plane of $\mathcal{G}^*$ containing $p$. Then $w$ is a subspace of $\mathcal{G}$ with 7 points and 6 lines that is configured as a plane.  To define $e(w)$ we consider the cases where there are 3 lines of $w$ containing $p$, and where only 2 lines of $w$ contain $p$. 

Suppose $w$ has 3 lines $\ell,m,n$ that contain $p$. If $d$ has opposite values on two of these lines, say $d(\ell)={\downarrow}$ and $d(m)={\uparrow}$, then by the definition of a direction of a pre-orthogeometry $\ell,m$ must belong to a plane $z$ in which $\ell,m$ are the only lines containing $p$. But the five points on the lines $\ell,m$ belong to both the planes $w,z$, and it is easily seen that since both $w,z$ are subspaces of $\mathcal{G}$ that the other two points of the plane $w$ belong to $z$. But this implies that $w=z$, a contradiction. We conclude that $d$ takes the same value of $\downarrow,\uparrow$ on all three of $\ell,m,n$. If this value is $\downarrow$ set $e(w)=(p,w)$, and if this value is $\uparrow$, set $e(w)=(w,p)$. If $w$ has only two lines $\ell,m$ that contain $p$, then the definition of a direction of a pre-orthogeometry shows that $d$ takes opposite values of $\downarrow,\uparrow$ on $\ell,m$. If $d(\ell)={\downarrow}$ and $d(m)={\uparrow}$, set $e(w)=(m,\ell)$, and if $d(\ell)={\uparrow}$ and $d(m)={\downarrow}$ set $e(w)=(l,m)$. 

We use \cite[Defn.~4.9]{HHLN} to show $e$ is a direction for $p$ in the orthodomain $\mathcal{G}^*$. Suppose $p\leq y$. By construction $e(y)$ is a principal pair for $p$ in the Boolean domain ${\downarrow}y$. Also, if $\ell,m$ are lines containing $p$ with $e(\ell)=(p,\ell)$ and $e(m)=(m,p)$, then $d(\ell)={\downarrow}$ and $d(m)={\uparrow}$. So by the definition of a direction of a pre-orthogeometry there is a plane $w$ containing $\ell,m$ in which these are the only lines containing $p$, and we have seen this is the only plane containing $\ell,m$. Thus $\ell\vee m = w$ and $w$ covers $\ell,m$. Finally, if $p\leq z < y$ and $e(y)=(u,v)$ we must show $e(z)=(u\wedge z,v\wedge z)$. If $z=p$ this is clear. So the only case of interest is when $z=\ell$ and $y=w$ is a plane. If there are 3 lines of $w$ containing $p$ and $d(\ell)={\downarrow}$, then $e(w)=(p,w)$ and $e(\ell)=(p,\ell)$ and our condition is verified. If there are 3 lines of $w$ containing $p$ with $d(\ell)={\uparrow}$, then $e(w)=(w,p)$ and $e(\ell)=(\ell,p)$, and again our condition is verified. If there are only two lines $\ell,m$ of $w$ containing $p$, then as we have seen $d$ takes opposite values on these lines. Suppose $d(\ell)={\downarrow}$ and $d(m)={\uparrow}$. Then $e(w)=(m,\ell)$, $e(\ell)=(p,\ell)$ and $e(m)=(m,p)$. Again our condition is verified. Thus $e$ is a direction of $\mathcal{G}^*$ and it clearly restricts to $d$. 
\end{proof}

We are now in a position to establish our characterization of orthogeometries, those configurations of points and lines that arise as the 4 and 8-element Boolean subalgebras of an orthomodular poset. In this formulation we use the notion of a triangle. 

\begin{definition} 
A triangle in an orthogeometry is a set of three points, any two of which belong to a line. A triangle is non-degenerate if the three points are distinct and do not all lie on the same line.
\end{definition} 

We will make use of several results in \cite{HHLN} which depend on the condition of orthodomains being \emph{proper}, i.e., orthodomains without maximal elements of height 1 or less. The corresponding notion for pre-orthogeometries is that every point is contained in some line. We will call such a pre-orthogeometry \emph{proper} as well. Clearly an orthodomain $X$ is proper if and only if its associated pre-orthogeometry $(P,L)$ is proper, where $P$ consists of the atoms of $X$ and $L$ consists of the elements covering some atom in $X$. We call an orthoalgebra $A$ proper if it does not have blocks of four or fewer elements. Then $A$ is proper if and only if its associated orthodomain $\B(A)$ of Boolean subalgebras of $A$ is proper. 

\begin{theorem} \label{orthogeometry}
A proper pre-orthogeometry $\mathcal{G}=(P,L)$ is an orthogeometry iff the following hold.

\begin{enumerate}
\item Each non-degenerate triangle is contained in a subspace that is a plane. 
\item Every point has a direction. 
\end{enumerate}
\end{theorem}
\begin{proof}

Suppose $\mathcal{G}$ is an orthogeometry. Then up to isomorphism, it is the set of 4 and 8-element Boolean subalgebras of an orthomodular poset $A$.
Suppose $p,q,r$ is a non-degenerate triangle of $\mathcal{G}$. Then there are pairwise orthogonal $a,b,c\in A$ with none equal to $0,1$ with $p=\{0,a,a',1\}$, $q=\{0,b,b',1\}$ and $r=\{0,c,c',1\}$. In an orthomodular poset, such pairwise orthogonal elements generate a 16-element Boolean subalgebra, a fact that is not true in general orthoalgebras. So this triangle is contained in a plane of $\mathcal{G}$. Thus (1) holds. 

For (2), we first show that $\mathcal{G}^*$ is isomorphic to the collection of at most 16-element Boolean subalgebras of $A$. Surely any 16-element Boolean subalgebra of $A$ has its 4-element subalgebras being the points of a subspace of $\mathcal{G}$ that are configured as a plane. So each 16-element Boolean subalgebra of $A$ gives a plane of $\mathcal{G}^*$. But by \cite[Prop.~6.7]{HHLN} any configuration of points in $\mathcal{G}$ in the form of a plane arises as the 4-element Boolean subalgebras of some 16-element Boolean subalgebra of $A$. Thus $\mathcal{G}^*$ is isomorphic to the poset of at most 16-element Boolean subalgebras of $A$. Then \cite[Thm.~4.16]{HHLN} has the consequence that each point of the orthodomain $\mathcal{G}^*$ has a direction, and then Lemma~\ref{dir} yields that each point of the pre-orthogeometry $\mathcal{G}$ has a direction. Thus (2) holds. 

For the converse, suppose that (1) and (2) hold for $\mathcal{G}$. By Lemma~\ref{lk} $\mathcal{G}^*$ is an orthodomain, which is clearly proper, since $\mathcal G$ is proper. By Lemma~\ref{dir}, each point of $\mathcal{G}^*$ has a direction. The basic element $\bot$ of any orthodomain always has a direction. Thus $\mathcal{G}^*$ is a short orthodomain (short meaning its height is at most 3 \cite[Defn.~5.13]{HHLN}) with enough directions (meaning each basic element has a direction). Then by \cite[Thm.~5.18]{HHLN} there is an orthoalgebra $A$ with $\mathcal{G}^*$ being isomorphic to the Boolean subalgebras of $A$ with at most 16 elements. Thus $\mathcal{G}$ is isomorphic to the poset of Boolean subalgebras of $A$ with at most 8 elements. 

It remains to show that the orthoalgebra $A$ constructed in the previous paragraph is an orthomodular poset. Orthomodular posets are characterized among orthoalgebras by the property that for orthogonal elements $a,b$ we have that $a\oplus b$ is the least upper bound of $a,b$ and not simply a minimal upper bound \cite[Prop. 1.5.6]{DvurPulm:Trends}. It is sufficient to show this under the assumption that $a,b\neq 0,1$. Suppose $c\in A$ is an upper bound of $a,b$. We must show $a\oplus b\leq c$. We may assume $c\neq 0,1$. Set $p=\{0,a,a',1\}$, $q=\{0,b,b',1\}$ and $r=\{0,c,c',1\}$. Since $a,b,c'$ are pairwise orthogonal, these three points form a non-degenerate triangle in $\mathcal{G}$. So by condition (1) there is a plane of $\mathcal{G}$ that contains them. Then $a,b,c$ lie in a 16-element Boolean subalgebra of $A$, and this yields $a\oplus b\leq c$ since in a Boolean algebra $a\oplus b$ is the join $a\vee b$.
\end{proof}

We turn to the matter of identifying orthogeometries that arise from orthomodular lattices, and from complete orthomodular lattices. If $\mathcal{G}$ is the orthogeometry for an orthomodular poset $A$, then with appropriate operations \cite[Defn.4.15]{HHLN} on the set $\oDir(\mathcal{G}^*)$ of directions of $\mathcal{G}^*$ we have \cite[Thm.~4.16]{HHLN} that $\oDir(\mathcal{G}^*)$ is an orthomodular poset isomorphic to $A$. By Lemma~\ref{dir} there is a bijective correspondence between the directions $\oDir(\mathcal{G})$ of $\mathcal{G}$ and those of $\mathcal{G}^*$. This provides an orthoalgebra structure on $\oDir(\mathcal{G})$ making it isomorphic to $A$. The following consequence of these results will be useful. 

\begin{proposition}\label{ds}
Let $\mathcal{G}$ be the orthogeometry of a proper orthomodular poset $A$. If $d$ is a direction for a point $p$ and $e$ is a direction for a point $q$, then $d\leq e$ iff $d=e$ or $p,q$ are distinct, both lie on  a line $\ell$, $d(\ell)={\downarrow}$, and $e(\ell)={\uparrow}$. The directions $0,1$ are the least and largest directions. 
\end{proposition}

\begin{center}
\begin{tikzpicture}
\draw[fill] (0,0) circle(.05); \draw[fill] (3,0) circle(.05);
\draw (.25,0) -- (2.75,0);
\node at (-.5,0) {$p$}; \node at (3.5,0) {$q$}; \node at (1.5,-.5) {$\ell$}; \draw[->] (1,.25)--(1,-.25); \draw[->] (2,-.25)--(2,.25);
\node at (1,.75) {$d$}; \node at (2,.75) {$e$};
\end{tikzpicture}
\end{center}

\begin{proof}
Let $d,e$ be directions of $\mathcal{G}$. Since $\mathcal{G}$ is an orthogeometry, Lemmas \ref{lk}, \ref{dir} give that the directions of $\mathcal{G}$ extend to directions $d^*,e^*$ of the orthodomain $\mathcal{G}^*$. In the orthomodular poset $\oDir(\mathcal{G}^*)$ we have $d^*\leq e^*$ iff $d^*\oplus (e^*)'$ is defined. By \cite[Defn.~4.15]{HHLN} this occurs iff $p,q$ are distinct points of a line $\ell$ with $d^*(\ell) = (p,\ell)$ and $(e^*)'(\ell)=(q,\ell)$. This is equivalent to having $d(\ell)=\,\downarrow$ and $e(\ell)=\,\uparrow$. 
\end{proof}

For a set $D$ of directions of a proper orthogeometry $\mathcal{G}$ and a direction $e$ of $\mathcal{G}$, we call $e$ a cone of $D$ if $e$ is an upper bound of $D$ in the natural ordering of $\oDir(\mathcal{G})$. 
Note that the direction $1$ is a cone of any set of directions; if $1$ is an element of $D$ then $1$ is the only cone of $D$; every direction is a cone of the empty set of directions; and a direction $e$ is a cone of $D$ iff it is a cone of $D\setminus\{0\}$. 
The case when $D$ is a singleton other than $0,1$ is given by Proposition~\ref{ds}. The remaining cases are addressed by the following. 

\begin{proposition}
For a proper orthogeometry $\mathcal G$, let $D$ be a non-empty set of directions with none equal to 0,1. For each $d\in D$ let $p_d$ be the point such that $d$ is a direction for $p_d$. Then $e$ is a cone of $D$ iff $e=1$ or $e$ is a direction for a point $q$ and the following hold for each $d\in D\setminus\{e\}$
\vspace{1ex}
\begin{enumerate}
\item $q$ is distinct from each $p_d$
\item there is a line $\ell_d$ containing $p_d$ and $q$
\item $d(\ell_d)=\,\downarrow$ and $e(\ell_d)=\,\uparrow$
\end{enumerate}
\end{proposition}

\begin{proof}
Assume $e\neq 1$. Clearly $e$ cannot be 0, hence it is a direction for a point $q$. If conditions (1)--(3) are satisfied, then by Proposition~\ref{ds} $e$ is a cone of $D$. Conversely, if $e$ is a cone, then by Proposition~\ref{ds} (1)--(3) hold. 
\end{proof}

Call $e$ a minimal cone for a set $D$ of directions if $e$ is the least upper bound of $D$ in $\oDir(\mathcal{G})$. If $e=1$ this means that $e$ is the only cone of $D$. If $e=0$ this means it is a direction of $D$, and this occurs iff any direction in $D$ is 0. Other cases are covered by the following.

\begin{proposition}
Let $D$ be a 
set of directions for a proper orthogeometry $\mathcal{G}$ and let a direction $e$ for a point $q$ be a cone of $D$. Assume that $e$ is neither $0$ or $1$. Then $e$ is a minimal cone for $D$ iff 0 is not a cone of $D$ and for any direction $f\neq e$ for a point $r\neq q$ that is a cone of $D$ there is a line $m$ containing $q,r$ and with $e(m)=\,\downarrow$ and $f(m)=\,\uparrow$. 
\end{proposition}
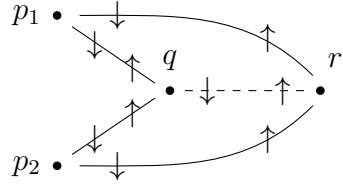
\vspace{1ex}

\begin{center}
\begin{tikzpicture}[scale = 1.0]
\draw[fill] (0,1) circle (.05); \draw[fill] (0,-1) circle (.05); \draw[fill] (1.5,0) circle (.05); \draw[fill] (3.5,0) circle (.05); \draw (.2,.85) -- (1.3,.1); \draw (.2,-.85) -- (1.3,-.1); \draw (.3,1) to [out=0,in=135] (3.4,.2); \draw (.3,-1) to [out=0,in=-135] (3.4,-.2); \draw[dashed] (1.7,0)--(3.3,0);
\node at (-.4,1) {$p_1$}; \node at (-.4,-1) {$p_2$}; \node at (1.5,.4) {$q$}; \node at (3.7,.4) {$r$}; \node at (2,0) {$\downarrow$}; \node at (3,0) {$\uparrow$};
\node at (.8,1) {$\downarrow$}; \node at (2.8,.7) {$\uparrow$};
\node at (.8,-1) {$\downarrow$}; \node at (2.8,-.65) {$\uparrow$};
\node at (.5,.6) {$\downarrow$}; \node at (1,.3) {$\uparrow$};
\node at (.5,-.65) {$\downarrow$}; \node at (1,-.3) {$\uparrow$};
\end{tikzpicture}
\end{center}

\begin{proof}
If $e$ is a minimal cone for $D$ then $0$ cannot be a cone for $D$, and for any $f$ as described we have $e\leq f$ so Proposition~\ref{ds} gives the existence of the line $m$ and the behavior of the directions $e,f$ at $m$. Conversely, if these conditions are satisfied we must show that $e\leq g$ for any cone $g$ of $D$. By assumption, $0$ is not a cone of $D$, clearly $e\leq e$, $e\leq 1$, 
and Proposition~\ref{ds} provides that $e\leq f$ for any $f$ as described in the statement. The remaining case, the direction $e'$ of $q$, cannot occur since this would entail that 
$e,e'$ are cones of $D$, hence $0$ is a cone for $D$. 
\end{proof}

Since minimal cones correspond to least upper bounds, we have the following. 

\begin{theorem} \label{complete}
A proper orthogeometry $\mathcal{G}$ is the orthogeometry of a (complete) orthomodular lattice iff any set of (at least) two directions has a minimal cone. 
\end{theorem}

We next extend matters to morphisms. A morphism of hypergraphs \cite[Defn.~7.2]{HHLN} is a partial mapping between their points satisfying certain conditions. These conditions involve the planes of these hypergraphs but can be adapted nearly verbatim to apply to orthogeometries, referring to the planes of the associated hypergraphs. We do not write this translation here, but provide the following more expedient version. 

\begin{definition} \label{orthogeometry morphism}
For orthogeometries $\mathcal{G}_1$ and $\mathcal{G}_2$, a morphism of orthogeometries $\alpha:\mathcal{G}_1\to\mathcal{G}_2$ is a partial function between their points such that this partial mapping is a hypergraph morphism between $\mathcal{G}_1^*$ and $\mathcal{G}_2^*$. 
\end{definition}

As in \cite{HHLN}, write $\alpha(p)=\bot$ when the partial mapping is not defined at $p$. A hypergraph morphism is called proper \cite[Defn.~7.8]{HHLN} when certain conditions involving the image of small sets 
apply. A morphism of orthogeometries will be called proper when the associated morphism of hypergraphs is proper. For $\alpha$ a proper hypergraph morphism $\alpha$, there is an orthoalgebra morphism $f_\alpha$ between the associated orthoalgebras of directions  \cite[Prop.~7.18]{HHLN}. An obvious translation of these results to the current setting gives the following. 


\begin{proposition}
Let $\alpha:\mathcal{G}_1\to\mathcal{G}_2$ be a proper morphism of proper orthogeometries. Then there is an orthoalgebra morphism $f_\alpha:\oDir(\mathcal{G}_1)\to\oDir(\mathcal{G}_2)$. 
\end{proposition}

This map $f_\alpha$ is described as in \cite[Defn.~7.15]{HHLN}. In essence, for a direction $d$ of a point $p$ and line $\ell$ through $p$ whose image is a line $\ell'$, the direction $f_\alpha(d)$ takes the same value at $\ell'$ as $d$ takes at $\ell$. So $f_\alpha$ is easily computed. Using this map $f_\alpha$ we can describe the orthogeometry morphisms that preserve binary and arbitrary joins. 

\begin{definition}
A proper orthogeometry morphism $\alpha:\mathcal{G}_1\to\mathcal{G}_2$ is (finite) join preserving if it takes a minimal cone $e$ of a (finite) set $D$ of directions to a minimal cone $f_\alpha(e)$ of $f_\alpha[D]$. With an eye to future use with AW*-algebras, we call an orthogeometry morphism that preserves all joins \emph{normal}. 
\end{definition}

By construction, the following is trivial. 

\begin{proposition} \label{lifts}
A proper orthogeometry morphism $\alpha$ between proper orthogeometries preserves (finite) joins iff the associated orthoalgebra morphism $f_\alpha$ preserves (finite) joins.
\end{proposition}

Theorems~\ref{orthogeometry} and \ref{complete} give a characterization of the sets $\mathcal{G}=(P,L)$ of points and lines that arise as the 4-element and 8-element Boolean subalgebras of an orthomodular poset, \mbox{orthomodular} lattice, and complete orthomodular lattice. Adapting \cite[Thm.~7.20]{HHLN} we have a characterization of orthomodular poset, orthomodular lattice, and complete orthomodular lattice homomorphisms in terms of the orthogeometry morphisms between them satisfying various additional properties. We continue this process for Boolean algebras below. 

\begin{theorem}
An orthogeometry $\mathcal{G}=(P,L)$ is the set of 4-element and 8-element Boolean subalgebras of a Boolean algebra iff any two points lie in a plane of $\mathcal{G}$. Proper orthogeometry morphisms between such orthogeometries correspond to a Boolean algebra homomorphisms. 
\end{theorem}

\begin{proof}
This is a simple consequence of the facts \cite{Kalmbach} that an orthomodular poset in which any two elements lie in a Boolean subalgebra is a Boolean algebra, and that the orthomodular poset homomorphisms between Boolean algebras are exactly the Boolean algebra homomorphism between them. 
\end{proof}

\begin{remark}
In the characterizations of sets $\mathcal{G}=(P,L)$ arising as orthomodular posets or as Boolean algebras, the ingredient that carries most weight is the existence of directions. This is a first order property in an appropriate language. So if $\mathcal{G}$ fails to have enough directions, then an application of the compactness theorem shows there is a finite reason for this failure. It would be interesting to see if a description of such $\mathcal{G}$ arising from Boolean algebras could be given without reference to directions, perhaps in terms of a finite number of forbidden configurations. This is an open problem. 
\end{remark}

To conclude, we place results in a categorical context. Let $\mathbf{OA}$ be the category of orthoalgebras and the morphisms between them and $\mathbf{OH}$ be the category of orthohypergraphs and the hypergraph morphisms between them. We recall that an orthohypergraph is a hypergraph that is isomorphic to the hypergraph of some orthoalgebra. The following was established in \cite{HHLN} but with the functor we call $\mathcal{H}$ here being called $\mathcal{G}$ there. 

\begin{theorem}
There is a functor $\mathcal{H}:\mathbf{OA}\to\mathbf{OH}$ that is essentially surjective on objects, injective on objects with the exception of the 1-element and 2-element orthoalgebras, and full and faithful with respect to proper orthoalgebra morphisms and proper hypergraph morphisms. 
\end{theorem}

Let $\mathbf{OMP}$ be the full subcategory of $\mathbf{OA}$ whose objects are orthomodular posets and $\mathbf{OG}$ be the category of orthogeometries isomorphic to ones arising from orthomodular posets and the orthogeometry morphisms between them. If $\mathcal{G}$ is such an orthogeometry, then $\mathcal{G}^*$ is an orthohypergraph with $(\mathcal{G}^*)_*=\mathcal{G}$. Since morphisms of orthogeometries are defined to be partial mappings that are orthohypergraph morphisms between their associated orthohypergraphs $\mathcal{G}^*$, the following is immediate. 

\begin{theorem} \label{jko} 
There is a functor $\mathcal{G}:\mathbf{OMP}\to\mathbf{OG}$ that is essentially surjective on objects, injective on objects with the exception of the 1-element and 2-element orthoalgebras, and full and faithful with respect to proper orthoalgebra morphisms and proper orthogeometry morphisms. 
\end{theorem}

Let $\mathbf{COML}$ be the category whose objects are complete orthomodular lattices and whose morphisms are ortholattice homomorphisms that preserve all joins. Let $\mathbf{COG}$ be the category of all orthogeometries isomorphic to ones arising from complete orthomodular lattices and the orthogeometry morphisms that preserve all joins. Note that these are non-full subcategories of $\mathbf{COML}$ and $\mathbf{OMP}$ respectively. Proposition~\ref{lifts} shows that the the functor $\mathcal{G}:\mathbf{OMP}\to\mathbf{OG}$ of Theorem~\ref{jko} restricts to a functor between these categories and provides the following. 

\begin{theorem}
The functor $\mathcal{G}:\mathbf{COML}\to\mathbf{COG}$ is essentially surjective on objects, injective on objects with the exception of the 1-element and 2-element orthoalgebras, and full and faithful with respect to proper morphisms in the two categories. 
\end{theorem}

\section{Operator algebras} In this section we lift results of the previous section to the setting of certain operator algebras known as AW*-algebras. We begin by recalling that a C*-algebra $M$ is a complete normed complex vector space with additional multiplication, unit $1_M$, and involution $*$, subject to certain well-known properties \cite{KR1}. We would like to emphasize that all our C*-algebras are assumed to be unital. A \emph{*-homomorphism} between C*-algebras is a linear map $\varphi:M\to N$ that preserves multiplication, involution, and units. We refer to \cite{KR1,Takesaki1} for details on operator algebras, and restrict ourselves to giving the appropriate definitions. 




\begin{definition}
A subset $N$ of a C*-algebra $M$ is a \emph{C*-subalgebra} of $M$ if it is a linear subspace that is closed under multiplication, involution, contains the unit $1_M$, and is additionally a closed subset in the metric topology induced by the norm. 
\end{definition}

We note that the condition that a C*-subalgebra contains the unit is important here and is not always assumed in the operator algebras literature.

\begin{definition}
Let $M$ be a C*-algebra, and let $x\in M$. Then $x$ is called \emph{self adjoint} if $x^*=x$; \emph{positive} if $x=y^2$ for some self adjoint $y\in M$; and a \emph{projection} if $x^2=x=x^*$. We denote the set of self-adjoint elements in $M$ by $M_\sa$, and the projections in $M$ by $\Proj(M)$.
\end{definition}	


A key fact here is the following, which holds more generally for projections of any *-ring. 

\begin{proposition}
For any C*-algebra $M$, its projections $\Proj(M)$ form an orthomodular poset when ordered by $p\leq q$ iff $pq=p$ and with orthocomplementation $p'=1_M-p$. Moreover, if two projections $p$ and $q$ commute, their join $p\vee q$ is given by $p+q-pq$.
\end{proposition}


The \emph{Jordan product} of a C*-algeba, a symmetrized version of the ordinary product, that is given by 
$$x\circ y=\frac{xy+yx}{2}$$
It is easily seen that a map $\varphi:M\to N$ that preserves linear structure and units preserves Jordan multiplication if and only if it preserves squares, and this provides an expedient way to define the appropriate morphisms. 

\begin{definition}
For C*-algebras $M$ and $N$, a \emph{Jordan homomorphism} $\varphi:M\to N$ is a linear map that preserves the unit, involution $*$, and squares: $\varphi(x^2)=\varphi(x)^2$. 
\end{definition}

Since a Jordan homomorphism preserves involution, it restricts to a morphism $M_\sa\to N_\sa$. Conversely \cite{KR1}, a linear map $M_\sa\to N_\sa$ that preserves units, the involution, and squares can uniquely be extended to a Jordan homomorphism $M\to N$. It is obvious that a *-homomorphism is a Jordan homomorphism, and that between commutative C*-algebras the notions coincide. A deeper look refines this result and is provided in the following. Here we recall \cite{Stormer} that a linear map $\varphi:M\to N$ between C*-algebras is \emph{positive} if it preserves positive elements, \emph{$n$-positive} if the associated map between matrix algebras $\mathrm{M}_n(\varphi):\mathrm{M}_n(M)\to\mathrm{M}_n(N)$ is positive, and \emph{completely positive} if it is $n$-positive for each $n$. 

\begin{lemma}\label{lem:Jordan homomorphism}
Suppose that $M$ and $N$ are C*-algebras. If $M$ is commutative, then any Jordan homomorphism $\varphi:M\to N$ is a *-homomorphism.
\end{lemma}	

\begin{proof}
By the Gelfand-Naimark theorem, we can embed $N$ into $B(H)$ for some Hilbert space $H$, and then view $\varphi$ as a Jordan homomorphism from $M$ to $B(H)$. Since $\varphi$ preserves squares, it is positive, and since $M$ is commutative, it follows from \cite[Thm. 1.2.5]{Stormer} that $\varphi$ is completely positive, hence certainly $2$-positive. By \cite[Thm. 3.4.4]{Stormer}, $\varphi$ is so-called \emph{non-extendible}. It follows from \cite[Thm. 3.4.5]{Stormer} that $\varphi$ is a *-homomorphism, which is clearly still a *-homomorphism if we restrict its codomain to $N$.	
\end{proof}

We turn to the class of C*-algebras that play the feature role here, AW*-algebras. These were introduced by Kaplansky \cite{Kaplansky} to give an algebraic formulation of the key concept of von Neumann algebras, also known as W*-algebras, that they have a large supply of projections and their projections form a complete orthomodular lattice. However, even in the commutative case AW*-algebras are much more general than their von Neumann counterparts. We begin with a definition close to Kaplansky's original. 

\begin{definition} \label{oppp}
A C*-algebra $M$ is an \emph{AW*-algebra} if its projections $\Proj(M)$ are a complete orthomodular lattice and each maximal commutative C*-subalgebra of $M$ is generated by its projections. 
\end{definition}
Here a C*-subalgebra $C$ of $M$ is said to be generated by a subset $S$ of if it is equal to the smallest C*-subalgebra of $M$ that contains $S$. This smallest C*-subalgebra, which we will denote by $C*(S)$, exists because the intersection of C*-subalgebras is a C*-subalgebra. Since we assumed that all C*-subalgebras of $M$ contain $1_M$, it follows that $C^*(S)$ will contain $1_M$ as well, even when $S$ does not contain the unit of $M$.

There is a substantially different way to get to AW*-algebras. A *-ring $A$ is a \emph{Baer*-ring} if for each non-empty subset $S$ of $A$ the \emph{right annihilator} 
\[R(S)=\{x\in A:sx=0\text{ for each }s\in S\}\] of $S$ is a principal right ideal of $A$ generated by a projection $p$, i.e., $R(S)=pA$. 
The following is given in \cite[Thm. III.1.8.2]{Blackadar}.

\begin{theorem}
A C*-algebra is a AW*-algebra if and only if it is a Baer*-ring. 
\end{theorem}

For an element $x$ in a Baer*-ring $A$ the right annihilator $R(\{x\})$ is a principal ideal generated by a projection $p$. We define the \emph{right projection} $RP(x)$ to be $1_A-p$. Using this, we have the following \cite[Defn's 4.3, 4.4]{Berberian}. 

\begin{definition}\label{subalgebra}
If $A$ is a Baer*-ring, then a *-subring $B$ of $A$ is a \emph{Baer*-subring} if (i) $x\in B$ implies $RP(x)\in B$ and (ii) every non-empty set of projections in $B$ has its supremum in the projection lattice of $A$ belong to $B$. If $M$ is an AW*-algebra, then an \emph{AW*-subalgebra} of $M$ is a C*-subalgebra $N$ of $M$ that is also a Baer*-subring. 
\end{definition}

There is an alternate description of AW*-subalgebras that is very useful \cite[Exercise 4.21]{Berberian}.




\begin{proposition}
Let $M$ be an AW*-algebra and $N$ be a C*-subalgebra of $M$. Then $N$ is an AW*-subalgebra of $M$ if and only if $N$ is an AW*-algebra and for any orthogonal set of projections in $N$ its join taken in the projection lattice of $M$ belongs to $N$. 
\end{proposition}
 

Primary examples of AW*-algebras are von Neumann algebras, those C*-subalgebras of the algebra $B(H)$ of operators on a Hilbert space $H$ that are equal to their double commutants in $B(H)$. Here the commutant of a subset $S$ of an AW*-algebra $M$ is the set $S'$ of all $a\in M$ such that $ab=ba$ for each $b\in S$. However, even in the commutative case AW*-algebras are more general than von Neumann algebras. 
A commutative C*-algebra is an AW*-algebra if and only if its Gelfand spectrum is a Stonean space, that is, an extremally disconnected compact Hausdorff space \cite{SaitoWright,SaitoWrightBook}. In contrast, the commutative C*-algebras that are von Neumann algebras are exactly those whose Gelfand spectrum is hyperstonean \cite{KR1}.

Combining Gelfand and Stone dualities, the category of commutative AW*-algebras and \mbox{*-homomorphisms} is equivalent via the functor $\Proj$ to the category of complete Boolean algebras and Boolean algebra homomorphisms between them. We introduce a new class of maps that preserve the infinite join structure. 
%


\begin{definition}\label{def:normal morphism}
Let $M$ and $N$ be AW*-algebras and let $\varphi:M\to N$ be either a Jordan homomorphism or a *-homomorphism, then we call $\varphi$ \emph{normal} if it preserves suprema of arbitrary collections of projections. 	
\end{definition}

The functor $\Proj$ gives an equivalence between the category of commutative AW*-algebras and normal *-homomorphisms and the category of complete Boolean algebras and Boolean algebra homomorphisms that preserve arbitrary joins. By \cite[Cor. 6.10]{Guram} these categories are dually equivalent to the category of Stonean spaces and open continuous functions. 

	We require additional facts about normal *-homomorphisms and normal Jordan homomorphisms between AW*-algebras. These roughly duplicate well-known results in the von Neumann algebra setting, but require different proofs and are difficult to find in the literature. Sources are the book of Berberian on Baer*-rings \cite{Berberian} and the paper of Heunen and Reyes \cite{HeunenReyes}. Most of these results are also included with proof in the PhD thesis of the second author, see \cite[\S 2.4]{Lindenhovius}.
We begin with the following found in \cite[Exercise 23.8]{Berberian}. 

\begin{lemma}\label{lem:normal morphisms preserve right projections}
Let $\varphi:M\to N$ be a normal *-homomorphism between AW*-algebras $M$ and $N$. Then $RP(\varphi(x))=\varphi(RP(x))$ for each $x\in M$.
\end{lemma}

So normal *-homomorphisms between AW*-algebras preserve joins of projections and right projections $RP(x)$. Also, by Definition~\ref{subalgebra},  AW*-subalgebras are exactly C*-subalgebras that are closed under joins of projections and right projections. Since the image and pre-image under a *-homomorphism of a C*-subalgebra is a C*-subalgebra, we have the following. 

\begin{lemma}\label{lem:image of normal morphism}
Let $\varphi:M\to N$ be a normal *-homomorphism between AW*-algebras $M$ and $N$. Then the image and pre-image of AW*-subalgebras is an AW*-subalgebra.
\end{lemma}	

The intersection of a family of AW*-subalgebras is again such \cite[Prop. 4.8]{Berberian}. So for any subset $S$ of an AW*-algebra, there is a smallest AW*-subalgebra $AW^*(S)$ generated by it. Using Lemma~\ref{lem:image of normal morphism} we obtain the following. 

\begin{lemma}\label{lem:homomorphisms preserve generating sets}
Let $\varphi:M\to N$ be a normal *-homomorphism between AW*-algebras and $S\subseteq M$. Then $\varphi[AW^*(S)]=AW^*(\varphi[S])$.
\end{lemma}


\begin{lemma}
Let $M$ and $N$ be AW*-algebras, and let $\varphi,\psi:M\to N$ be normal Jordan homomorphisms. If $\varphi$ and $\psi$ coincide on $\Proj(M)$, then $\varphi=\psi$.
\end{lemma}	
\begin{proof}
We first show that $\varphi$ and $\psi$ coincide on any commutative AW*-subalgebra $C$ of $M$. By Lemma~\ref{lem:Jordan homomorphism} the restrictions of $\varphi$ and $\psi$ to $C$ are *-homomorphisms which are clearly normal. So by Lemma \ref{lem:image of normal morphism} both $\varphi[C]$ and $\psi[C]$ are AW*-subalgebras of $N$, and these AW*-subalgebras are clearly commutative. As mentioned above Definition \ref{def:normal morphism}, the functor $\Proj$ is a an equivalence of categories of commutative AW*-algebras and complete Boolean algebras. Therefore, a *-homomorphism on $C$ is completely determined by its restriction to $\Proj(C)$, hence $\varphi$ and $\psi$ must coincide on $C$.
Let $x\in M$ be self-adjoint. Then $\{x\}''$ is commutative by \cite[Prop. 3.9]{Berberian} and an AW*-subalgebra of $M$ by \cite[Prop. 4.8(iv)]{Berberian}. It follows that $\varphi(x)=\psi(x)$. Now, let $x\in M$ be arbitrary. Then $x=x_1+ix_2$ with $x_1=(x+x^*)/2$ and $x_2=(x-x^*)/2i$, both of which are self adjoint. Hence \[\varphi(x)=\varphi(x_1)+i\varphi(x_2)=\psi(x_1)+i\psi(x_2)=\psi(x).\qedhere \] 
\end{proof}			

Before the key results, we require some further terminology that is standard in the area \cite{KR1}. Here, for a commutative C*-algebra $C$, the matrix algebra of $2\times 2$ matrices with coefficients in $C$ is $\mathrm{M}_2(C)$. A standard result is that this is a C*-algebra, and is an AW*-algebra if $C$ is such. 

\begin{definition}
An AW*-algebra $M$ is of \emph{type} $I_2$ if there is a commutative AW*-algebra $C$ such that $M$ is *-isomorphic to $\mathrm M_2(C)$. We say that $M$ has a type I$_2$ summand if it is *-isomorphic to the direct sum $N_1\oplus N_2$ of some AW*-algebras $N_1$ and $N_2$, where $N_2$ is a type I$_2$ AW*-algebra. 
\end{definition}

In \cite[Thm 4.2]{Hamhalter2} Hamhalter showed that Dye's Theorem can be extended to the class of AW*-algebras. Combining this with the previous lemma to give uniqueness yields the following.   

\begin{theorem}\label{thm:hamhalter}
Let $M$ and $N$ be AW*-algebras and assume $M$ has no type I$_2$ summands. Then any ortholattice homomorphism $\psi:\Proj(M)\to\Proj(N)$ that preserves arbitrary joins uniquely extends to a normal Jordan homomorphism $\varphi:M\to N$. 
\end{theorem}	


We now make use of these results. Let $\mathbf{AW^*_J}$ be the category of AW*-algebras and normal Jordan homomorphisms. Recall that $\mathbf{COML}$ is the category of complete orthomodular lattices and ortholattice morphisms that preserve all suprema. Let $\mathbf{AWOML}$ be its full subcategory of orthomodular lattices isomorphic to the projection lattice of an AW*-algebra. For any C*-algebra $M$ we have that $\Proj(M)$ is an orthomodular poset, and if $M$ is an AW*-algebra, then by definition $\Proj(M)$ is a complete orthomodular lattice. For a normal Jordan homomorphism $\varphi:M\to N$ between AW*-algebras, it is obvious that $\varphi$ maps projections to projections, and by normality it preserves arbitrary joins of projections. Since $\varphi$ is linear and $p'=1-p$ it is clear that $\varphi$ preserves orthocomplementation. It follows that restricting $\varphi$ to $\Proj(M)$ gives a complete ortholattice homomorphism from $\Proj(M)$ to $\Proj(N)$. This yields the following. 

\begin{theorem}
There is a functor $\Proj:\mathbf{AW^*_J}\to\mathbf{AWOML}$ that takes an AW*-algebra to its projection lattice and a normal Jordan homomorphism to its restriction to the projection lattices. 
\end{theorem}

This functor has additional properties. By definition, it is surjective on objects. If we restrict attention to AW*-algebras without type $I_2$ factor and their corresponding projection lattices we have the following. For two such AW*-algebras $M$ and $N$ we have $\Proj(M)$ is isomorphic to $\Proj(N)$ if and only if $M$ and $N$ are Jordan isomorphic, a fact that follows from Theorem~\ref{thm:hamhalter}. Further, there is a bijective correspondence between the normal Jordan homomorphisms between $M$ and $N$ without type $I_2$ factor and complete ortholattice homomorphisms between their projection lattices. This also follows from Theorem~\ref{thm:hamhalter} and the trivial fact that a normal Jordan homomorphism between $M$ and $N$ restricts to a complete ortholattice homomorphism between their projection lattices. 

\begin{definition}\label{def:proper C*-algebra}
A C*-algebra $M$ is \emph{proper} if it is not *-isomorphic to either $\mathbb{C}^2$ or to $\mathrm{M}_2(\mathbb C)$. 
\end{definition}

By \cite[Prop. 3.3]{Hamhalter} a C*-algebra $M$ is proper if and only if it does not have any maximal commutative C*-subalgebras of dimension 2, which occurs if and only if $\Proj(M)$ does not have any maximal Boolean subalgebras with 4 elements. So a C*-algebra $M$ is proper if and only if $\Proj(M)$ is proper, and this occurs if and only if the orthogeometry associated with $\Proj(M)$ is proper. 

A morphism between orthoalgebras is proper if it satisfies a somewhat awkward condition relating to the image of certain blocks not being small, i.e., if the images contain more than four elements. This could be easily translated into an equally awkward condition on normal Jordan homomorphisms, but we take the more expedient route. 

\begin{definition}
A normal Jordan homomorphism $\varphi:M\to N$ between AW*-algebras is \emph{proper} if its restriction to an ortholattice homomorphism between projection lattices is proper. 
\end{definition}

We obtain the following about the composition of the functors $\Proj$ and $\mathcal{G}$.

\begin{theorem}\label{thm:mainGProj}
The composite $\mathcal{G}\circ\Proj : \mathbf{AW^*_J}\to\mathbf{COG}$ is injective on proper AW*-algebras with no type $I_2$ factor. If $M$ and $N$ are proper AW*-algebras with no type $I_2$ factor, then there is a bijective correspondence between the proper normal Jordan homomorphisms from $M$ to $N$ and the proper normal orthogeometry morphisms from $\mathcal{G}(\Proj M)$ to $\mathcal{G}(\Proj N)$. 
\end{theorem}

We wish to treat this functor more directly in terms of AW*-algebras. 

\begin{definition}
Let $M$ be an AW*-algebra. Then we denote the set of all commutative AW*-subalgebras of $M$ by $\A(M)$, which we order by inclusion.
\end{definition}

Posets of commutative subalgebras of operator algebras have been studied before, for instance in \cite{Doring} where the poset $\V(M)$ of commutative von Neumann subalgebras of a von Neumann algebra $M$ is considered. Since any von Neumann algebra is an AW*-algebra, and the AW*-subalgebras of a von Neumann algebra $M$ are the von Neumann subalgebras of $M$, we obtain $\V(M)=\A(M)$. The poset $\CC(M)$ of commutative C*-subalgebras of a C*-algebra $M$ is studied in \cite{Hamhalter,Hamhalter2,Heunen,HeunenLindenhovius,Lindenhovius2,Lindenhovius} and is in general larger than $\A(M)$ in case $M$ is an AW*-algebra.




\begin{lemma}\label{lem:commutative subset generates commutative subalgebra}
Let $M$ be an AW*-algebra and $S\subseteq M$ be closed under involution and consist of mutually commuting elements. Then $AW^*(S)$ is a commutative AW*-subalgebra of $M$. 
\end{lemma}

\begin{proof}
It follows from \cite[Prop. 3.9]{Berberian} that $S\subseteq S''$, where $S''$ is a commutative AW*-subalgebra of $M$. So $S$ is contained in some commutative AW*-subalgebra of $M$, hence it must generate a commutative AW*-subalgebra.
\end{proof}



Lemmas \ref{lem:commutative subset generates commutative subalgebra}, \ref{lem:Jordan homomorphism} and \ref{lem:homomorphisms preserve generating sets} give the following. 

\begin{lemma}\label{lem:jordan homomorphisms preserve generating sets}
Let $\varphi:M\to N$ be a normal Jordan homomorphism between AW*-algebras and $S\subseteq M$ be a subset that is closed under the involution and that consists of mutually commuting elements. Then $\varphi[AW^*(S)]=AW^*(\varphi[S])$ and this is a commutative AW*-subalgebra of $N$. 
\end{lemma}


\begin{proposition}\label{prop:sups in AM}
Let $M$ be an AW*-algebra, and $\D\subseteq\A(M)$. Then $\bigvee\D$ exists in $\A(M)$ if and only if $\bigcup\D$ consists of mutually commuting elements, and in this case $\bigvee\D=AW^*\left(\bigcup\D\right)$. 
\end{proposition}		

\begin{proof}
Assume that $S=\bigcup\D$ consists of mutually commuting elements. Clearly $S$ is closed under the involution, so by Lemma \ref{lem:commutative subset generates commutative subalgebra} $AW^*(S)$ is a commutative. 
Let $C\in\A(M)$ be such that $D\subseteq C$ for each $D\in\D$. Then $S$ is contained in $C$, so $AW^*(S)\subseteq C$. Thus $AW^*(S)$ is the supremum of $\D$ in $\A(M)$. Conversely, assume that $\bigvee\D$ exists in $\A(M)$. Then $S\subseteq \bigvee\D$ and $\bigvee\D$ is a commutative AW*-subalgebra, so all elements in $\bigcup\D$ commute. 
\end{proof}

Suppose that $\varphi:M\to N$ is a normal Jordan homomorphism between AW*-algebras. We define $\A(\varphi):\A(M)\to\A(N)$ to be the map taking $C$ to $\varphi[C]$. We recall that a directedly complete partial order (dcpo) is a poset where every directed subset has a join and a Scott continuous map between dcpo's is a map that preserves directed joins. 

\begin{proposition} 
$\A:\mathbf{AW^*_J}\to\mathbf{DCPO}$ is a functor from the category of AW*-algebras and normal Jordan homomorphisms to the category of dcpo's and Scott continuous maps. 
\end{proposition}

\begin{proof}
Assume $\D\subseteq\A(M)$ is directed and set $S=\bigcup\D$. Then any $x,y\in S$ belong to some member of $\D$, so $S$ is commutative. So by Proposition \ref{prop:sups in AM} $\bigvee\D$ exists. Thus $\A(M)$ is a dcpo. Let $\varphi:M\to N$ be a normal Jordan homomorphism. Then by Proposition~\ref{prop:sups in AM} we have $\varphi(\bigvee\D) = \varphi[AW^*(S)]$ and $\bigvee\varphi[\D]=AW^*(\varphi[S])$. Lemma~\ref{lem:jordan homomorphisms preserve generating sets} gives that these are equal. 
\end{proof}	

\begin{remark}\label{nkl}
We can show more than stated in the result above. If $\D\subseteq\A(M)$ has a join, then $S=\bigcup\D$ is commutative, and the proof of the previous result shows $\varphi(\bigvee\D)=\bigvee\varphi[\D]$. So $\varphi$ not only preserves directed joins, it preserves all existing joins. 
\end{remark}

Let $A$ be an orthomodular lattice and $\mathcal{B}(A)$ be its poset of Boolean subalgebras. Directed joins in $\mathcal{B}(A)$ are given by unions and for an ortholattice homomorphism $f:A\to A'$ there is a Scott continuous map $\mathcal{B}(f):\mathcal{B}(A)\to\mathcal{B}(A')$ given by $\mathcal{B}(f)(S)=f[S]$. This gives the following.

\begin{proposition}
$\mathcal{B}:\mathbf{OML}\to\mathbf{DCPO}$ is a functor from the category of orthomodular lattices and ortholattice homomorphisms to the category of dcpo's and Scott continuous maps. 
\end{proposition}

For a complete orthomodular lattice $A$, we let $\mathcal{B}_C(A)$ be its poset of complete Boolean subalgebras. These are Boolean subalgebras of $A$ that are closed under arbitrary joins in $A$. Clearly $\mathcal{B}_C(A)$ is a subposet of $\mathcal{B}(A)$ and for each $S\in\mathcal{B}(A)$ there is a least member $\overline{S}$ of $\mathcal{B}_C(A)$ above it, the closure of $S$ under arbitrary joins and meets in $A$, due to the fact that the maximal elements of $\B(A)$ are complete Boolean subalgebras as follows from the remarks below \cite[Prop. 3.4 \& Lem. 4.1]{Kalmbach}. If $S$ is finite, $\overline{S}=S$, and it follows that the elements of $\mathcal{B}(A)$ of finite height are exactly the elements of $\mathcal{B}_C(A)$ of finite height. In particular, the atoms of $\B_C(A)$ are the Boolean subalgebras $B_p=\{0,p,p',1\}$ for some non-trivial $p\in A$.
The join of an updirected set $\D$ in $\mathcal{B}_C(A)$ is given by $\overline{\bigcup\D}$. In particular $\mathcal{B}_C(A)$ is a dcpo. Clearly any $S\in\B_C(A)$ satisfies $S=\bigvee\{B_p:p\in S\}$, hence $\B_C(A)$ is atomistic. If $f:A\to M$ is an ortholattice homomorphism between complete ortholattices that preserves arbitrary joins, then $f[\overline{S}]=\overline{f[S]}$ for any Boolean subalgebra $S$ of $A$, so $\mathcal{B}_C(f):\mathcal{B}_C(A)\to\mathcal{B}_C(M)$ given by $f(S)=f[S]$ preserves directed joins. This gives the following. 

\begin{proposition}
$\mathcal{B}_C:\mathbf{COML}\to\mathbf{DCPO}$ is a functor.
\end{proposition}

Let $M$ be an AW*-algebra with projection lattice $A$. Projections $p,q$ are orthogonal in $A$ iff $p\leq q'$, and this is easily seen to be equivalent to $pq=0=qp$. We have seen that if projections $p,q$ commute, then they belong to a Boolean subalgebra of $A$. Conversely, if $p,q$ belong to a Boolean subalgebra $B$ of $A$, then there are pairwise orthogonal $e,f,g$ in $B$ with $p=e+f$ and $q=f+g$. It follows that $pq=f=qp$. So projections commute iff they belong to a Boolean subalgebra of the projection lattice. So by Lemma~\ref{lem:commutative subset generates commutative subalgebra} there is a map $AW^*(\,\cdot\,)$ from the Boolean subalgebras of $A$ to the commutative AW*-subalgebras of $M$. 

\begin{proposition}
Let $M$ be an AW*-algebra with $A$ its projection lattice. Then the maps $\Proj$ and $AW^*(\,\cdot\,)$ are mutually inverse order-isomorphisms between the dcpo's $\A(M)$ and $\mathcal{B}_C(A)$. 
\end{proposition}

\begin{proof}
Clearly $\Proj$ and $AW^*(\,\cdot\,)$ are order preserving. Let $C$ be a commutative AW*-subalgebra of $M$. Then $C$ is an AW*-algebra, so by Definition~\ref{oppp} $C$ is generated as a C*-algebra by its projections, hence it is generated as an AW*-algebra by its projections. Thus $AW^*(\,\cdot\,)\circ\Proj$ is the identity. 

Let $B$ be a complete Boolean subalgebra of~$A$, and set $D=AW^*(B)$ and $C=\Proj(D)$. Then $B$ is a Boolean subalgebra of $C$ and since arbitrary joins in both $B$ and $C$ agree with those in $A$, we have that $B$ is a complete Boolean subalgebra of $C$. The Stone space $Y$ of $B$ is Stonean, so $C(Y)$ is an AW*-algebra. Let $i:\Proj(C(Y))\to B$ be the obvious isomorphism. Since $C(Y)$ is commutative, it does not have type $I_2$ summands, and since $B$ is a complete subalgebra of $C$, we have that $i:\Proj(C(Y))\to\Proj(D)$ preserves arbitrary joins. So by Theorem~\ref{thm:hamhalter}, $i$ extends to a normal Jordan homomorphism $\varphi:C(Y)\to D$, and as these are commutative algebras, $\varphi$ is a normal *-homomorphism. The image of $\varphi$ is an AW*-subalgebra that contains $B$, so $\varphi$ is onto. By Lemma~\ref{lem:normal morphisms preserve right projections}, $\varphi$ preserves right projections. So if $\varphi(x)=0$, then $\varphi(RP(x))=RP(\varphi(x))=RP(0)=0$. But $RP(x)$ is a projection, and since $i$ is an isomorphism, $RP(x)=0$, and this implies that $x=0$ \cite[Prop.~3.6]{Berberian}. Thus $\varphi$ is a *-isomorphism, and it follows that the projections of $AW^*(B)$ are exactly $B$. Thus $\Proj\circ AW^*(\,\cdot\,)$ is the identity. 
\end{proof}

\begin{corollary}
For an AW*-algebra $M$, the elements of $\A(M)$ of height $n$ are exactly the $(n+1)$-dimensional commutative C*-subalgebras of $M$. 
\end{corollary}

\begin{proof}
The $n$-dimensional commutative C*-algebras are those isomorphism to $C(X)$ for some set $X$ with $n$ elements. Each of these is an AW*-algebra. 
\end{proof}

\begin{corollary}\label{lem:A(M) atomistic}
Let $M$ be an AW*-algebra. Then $\A(M)$ is atomistic; its atoms are of the form $A_p=\mathrm{span}(p,1_M-p)$ for non-trivial projection $p\in M$. 
\end{corollary}

\begin{definition}
For an AW*-algebra $M$ and orthomodular lattice $A$, let $\A_*(M)$ be the elements of height at most two in $\A(M)$ and let $\mc{B}_*(A)$ be the elements of height at most two in $\mathcal{B}(A)$. 
\end{definition}

For an orthomodular lattice $A$, its associated orthogeometry $\mathcal{G}(A)$ is constructed as the pair $(P,L)$ where $P$ is the set of atoms of $\mathcal{B}(A)$ and $L$ is the set of elements of height two in $\mathcal{B}(A)$. So one may naturally consider $\mathcal{B}_*(A)$ as giving the orthogeometry $\mathcal{G}(A)$, and similarly $\A_*(M)$ as giving the orthogeometry for the projection lattice of $M$. 

\begin{corollary}\label{prop:A orthogeometry}
Let $M$ be an AW*-algebra. Then $\A_*(M)$ is an orthogeometry isomorphic to $\mathcal G(\Proj(M))$, which is proper if and only if $M$ is proper.
\end{corollary}

\begin{proof}
The only part that is not trivial is properness, and this follows from the remark below Definition \ref{def:proper C*-algebra}. 
\end{proof}	

Recall that a morphism between orthogeometries is a partial function between their points satisfying certain conditions.
Let $A$ and $A'$ be complete orthomodular lattices. Given an ortholattice morphism $f:A\to A'$, the map $\B(f):\B(A)\to\B(A')$ restricts to a partial map from the atoms $B_p=\{0,p,p',1\}$ of $\mathcal{B}(A)$ to the atoms of $\mathcal{B}(A')$. Explicitly, this is given by
\[ \mathcal{B}(f)(B_p)= \begin{cases}
B_{f(p)}, & \varphi(p)\neq 0,1;\\
\perp, & \text{otherwise}.
\end{cases}
\]
This restriction of $\B(f)$ is precisely the orthogeometry morphism $\mathcal{G}(f):\mathcal G{(A)}\to\mathcal G(A')$ as is shown in \cite[Def. 7.5, Prop. 7.6]{HHLN}.


\begin{proposition}\label{prop:A isomorphic to GProj}
There is a functor $\A_*:\mathbf{AW^*_J}\to \mathbf{COG}$ taking an AW*-algebra $M$ to the orthogeometry $\A(M)$ and a normal Jordan homomorphism $\varphi:M\to N$ to the partial function obtained as the restriction of $\A(\varphi)$ to the atoms of $\A_*(M)$. Further, this functor is naturally isomorphic to $\mathcal G\circ\Proj$. 
\end{proposition}

\begin{proof}
For an AW*-algebra $M$ the isomorphism $\Proj:\A_*(M)\to\mathcal{B}_*(\Proj (M))$ provides the desired natural isomorphism. 
\end{proof}

In combination with Theorem \ref{thm:mainGProj} we obtain:

\begin{corollary}\label{cor:maincor}
The functor $\A_* : \mathbf{AW^*_J}\to\mathbf{COG}$ is injective on proper AW*-algebras with no type $I_2$ factor. If $M$ and $N$ are proper AW*-algebras with no type $I_2$ factor, then $\varphi\mapsto\A_*(\varphi)$ is a bijective correspondence between the proper normal Jordan homomorphisms from $M$ to $N$ and the proper normal hypergraph morphisms from $\A_*(M)$ to $\A_*(N)$.
\end{corollary}

\begin{theorem}\label{thm:A of M main}
Let $M$ and $N$ be proper AW*-algebras, and assume that $M$ has no type I$_2$ summand. Then there exists a bijection $\varphi\mapsto\A_*(\varphi)$ between proper normal Jordan homomorphisms $\varphi:M\to N$ and maps $\Phi:\A(M)\to\A(N)$ that preserve all existing suprema, and that restrict to proper normal morphisms of orthogeometries $\A_*(M)\to\A_*(N)$.
\end{theorem}	

\begin{proof}
By Corollary \ref{cor:maincor}, the assignment $\varphi\mapsto\A_*(\varphi)$ is a bijection between proper normal Jordan homomorphisms $M\to N$ and proper normal morphisms of orthogeometries $\A_*(M)\to\A_*(N)$. Since $\A_*(\varphi)$ is the restriction of $\A(\varphi)$ to an orthogeometry morphism $\A_*(M)\to\A_*(N)$, it is sufficient to show that any proper normal morphism of orthogeometries $\Phi:\A_*(M)\to\A_*(N)$ uniquely extends to a map $\A(M)\to\A(N)$ that preserves arbitrary existing joins.
	
Since any such a $\Phi$ equals $\A_*(\varphi)$ for some proper normal Jordan homomorphism, $\A(\varphi)$ is such an extension. Let $\Psi$ be another extension. Note that $\A(\varphi)$ and $\Psi$ both preserve existing joins by Remark~\ref{nkl}. Since $\A(M)$ and $\A(N)$ are atomistic and these maps coincide on $\A_*(M)$, hence the set of atoms in $\A(M)$, it follows that they coincide on $\A(M)$. 
\end{proof}

\section{Concluding Remarks}
We have found a functor $\mathcal G\circ\Proj:\mathbf{AW^*_J}\to\mathbf{COG}$ assigning to an AW*-algebra an orthogeometry. This functor is injective on proper AW*-algebras without a type I$_2$ summand. Given proper AW*-algebras without type I$_2$ summands $M$ and $N$. We also obtained a bijection between proper normal Jordan homomorphisms between from $M$ to $N$, and proper normal orthogeometry morphisms from $\mathcal G(\Proj(M))$ to $\mathcal G(\Proj(N))$. Furthermore, we have shown that $\mathcal G(\Proj(M))$ is isomorphic to the poset $\A_*(M)$ of commutative AW*-subalgebras of $M$ dimension at most 2, hence the poset $\A(M)$ of all commutative AW*-subalgebras of $M$ contains the same information as $\mathcal G(\Proj(M))$. Indeed, we showed that the set of proper normal Jordan homomorphisms between $M$ and $N$ is bijective to the set of all maps from $\A(M)$ to $\A(N)$ that preserve all existing suprema and that restrict to proper normal morphisms of orthogeometries $\A_*(M)\to\A_*(N)$.

Given an AW*-algebra $M$, we can also consider the poset $\CC(M)$ of commutative C*-subalgebras, which in general will be larger than $\A(M)$ unless $M$ is finite dimensional. Then we can define $\CC_*(M)$ the subset of all elements of height at most two in $\CC(M)$, which is precisely $\A_*(M)$, so Corollary \ref{cor:maincor} holds also if we replace $\A$ by $\CC$. However, since the poset $\CC(M)$ is not atomistic, as follows from \cite[Thm. 2.4, Thm. 5.5., Thm 9.7]{HeunenLindenhovius}, we cannot find a bijection between Jordan homomorphisms from $M$ to another AW*-algebra $N$ and maps from $\CC(M)$ to $\CC(N)$ as in Theorem \ref{thm:A of M main}.

We already remarked that for an von Neumann algebra $M$ the poset $\V(M)$ coincides with $\A(M)$, hence all our statements hold as well if we replace the class of AW*-algebras by the class of von Neumann algebras. One could raise the question whether for an arbitrary AW*-algebra $M$ we can `recognize' from $\A_*(M)$ whether or not $M$ is a von Neumann algebra. By \cite{Pedersen} von Neumann algebras are characterized as the AW*-algebras $M$ with a separating family of normal states. Here a \emph{state} is a positive functional $\omega:M\to\mathbb C$ such that $\omega(1)=1$; a family $\mathcal F$ of states on $M$ is \emph{separating} if for each nonzero self adjoint $a\in M$ there is some $\omega\in\mathcal F$ such that $\omega(a)\neq 0$. One would like to proceed by recover normal states of $M$ as morphisms from $\A_*(M)$ to $\A_*(\mathbb C)$, but this is not possible for two reasons: firstly, states are in general not Jordan homomorphisms, and secondly, $\A_*(\mathbb C)$ is empty, hence certainly not proper. So we do not see how we can recognize from $\A_*(M)$ directly whether or not $M$ is a von Neumann algebra. However, we have the following indirect result: if $\A_*(M)$ is isomorphic to $\A_*(N)$ for some von Neumann algebra $N$, then $M$ should also be a von Neumann algebra, since the isomorphism between $\A_*(M)$ and $\A_*(N)$ implies the existence of some Jordan isomorphism $\varphi:M\to N$. Then $\mathcal F$ consisting of states $\omega\circ\varphi$ for a normal state $\omega$ on $N$ turns out to be a family of normal states on $M$ that is separating. 

In \cite{HeunenReyes}, a complete invariant for AW*-algebras with normal *-homomorphisms, called \emph{active lattices} was introduced, which consists of the othomodular lattice $\Proj(M)$ of projections of an AW*-algebra $M$ together with a group action on $\Proj(M)$. By the results of \cite{HHLN} the orthomodular lattice part of an active lattice associated to $M$ contains the same information as the orthogeometry $\mathcal G(\Proj(M))$, and as a consequence our work shows that orthogeometries encode the Jordan structure of $M$, and that the group action encodes the extra information that is required to obtain an invariant for *-homomorphisms instead of Jordan homomorphisms. 

Finally, both the results of orthogeometries and active lattices rely on the fact that AW*-algebras have abundant projections. This begs the question whether we can extend the results in this contribution to a larger class of operator algebras with ample projections, for instance the real rank zero algebras. A possible solution for this problem is to generalize Theorem \ref{thm:hamhalter} to real rank zero algebras. This is still an open problem.

\section*{Acknowledgements}
We thank Chris Heunen and Mirko Navara, since this work would have been impossible without them.
Bert Lindenhovius was funded by the AFOSR under the MURI
grant number FA9550-16-1-0082 entitled, "Semantics, Formal Reasoning, and Tool
Support for Quantum Programming".

\end{document}